\documentclass{amsart}
\usepackage{amssymb,amsmath,amsrefs}
\usepackage[colorlinks=true]{hyperref}
\hypersetup{urlcolor=blue, citecolor=green}
\usepackage[dvips]{graphicx}
\usepackage{color}

\usepackage{mathrsfs}
\usepackage{tikz}
\usetikzlibrary{arrows.meta, positioning}

\theoremstyle{plain}

\newtheorem{mainthm}{Theorem}

\newtheorem*{conj*}{Conjecture}
\newtheorem*{cor*}{Corollary}
\newtheorem{theorem}{Theorem}[section]

\newtheorem{proposition}[theorem]{Proposition}

\newtheorem*{notation}{Notation}
\newtheorem{lemma}[theorem]{Lemma}

\theoremstyle{definition}
\newtheorem*{def*}{Definition}
\newtheorem{remark}[theorem]{Remark}

\newtheorem{definition}[theorem]{Definition}

\newcommand{\eps}{\varepsilon}
\renewcommand{\epsilon}{\varepsilon}

\newcommand{\Z}{\mathbb{Z}}
\newcommand{\N}{\mathbb{N}}

\newcommand{\diam}{\operatorname{diam}}

\def \diam {\mbox{diam}}

\title[L-shadowing for the induced hyperspace homeomorphism]{L-shadowing for the induced hyperspace homeomorphism}

\author[Mayara Antunes, Bernardo Carvalho, Welington Cordeiro]{Mayara Antunes, Bernardo Carvalho, Welington Cordeiro}
\date{\today}



\begin{document}
\begin{abstract}
We prove that a homeomorphism $f\colon X\to X$ of a compact metric space satisfies the L-shadowing property if and only if its induced hyperspace homeomorphism $2^f$ also satisfies the L-shadowing property. In the proof, assuming only the L-shadowing property, we obtain the existence of points in the asymptotic local-product-structure with iterates approaching in a uniform rate of convergence to zero. This contrasts with the lack of uniformity of contraction on local stable/unstable sets on many homeomorphisms with the L-shadowing property.
\end{abstract}

\maketitle

\section{Introduction and statement of results}

The study of the shadowing property in dynamical systems can be viewed as an attempt to deal with approximations of orbits, usually called pseudo-orbits, which appear naturally in many situations. The pseudo-orbits do not necessarily represent the real evolution of the system, in general being associated with simulations of dynamical systems. The existence of an orbit shadowing a given pseudo-orbit means that the errors can be disregarded, the simulation is faithful, and that the pseudo-orbit indeed follows the system's evolution. The importance of the shadowing property comes from distinct research directions, such as the stability theory \cite{Walterspotp}, recurrence theory \cite{Conley}, chaotic and hyperbolic dynamics \cite{Anosov}, \cite{Bowen}, \cite{SMALE67}, and it is a significant part of the qualitative study of dynamical systems that contains several important and deep results (see the monographs \cite{Palmer} and \cite{PilyuginShadowingBook1999}).

Given a topological dynamical system $(X,f)$ formed by a compact metric space $X$ and a continuous map $f\colon X\to X$, its hyperspace dynamics is the dynamical system $(2^X, 2^f)$ where $2^X$ is the hyperspace of compact subsets of $X$ and $2^f$ is the induced map defined by $2^f(A)=f(A)$. The study of dynamics in hyperspaces comes from the idea of understanding collective dynamics. For example, if one considers population dynamics, then the hyperspace dynamics determines the collective behavior of sub-populations as an aggregate, opposed to the behavior of an individual agent. Hyperspace dynamics has applications in various disciplines of sciences, economics, and engineering (see, for example, \cite{DD}, \cite{MLR}, \cite{Rohde}, \cite{Rubinov}, \cite{Handwheel}, \cite{Electrons}, \cite{Zaslavski}). There is an extensive literature that investigates similarities and differences between the dynamics of a map and of its induced hyperspace map. The properties studied include topological entropy, sensitivity to initial conditions, recurrence properties, and specification-like properties (see \cite{KwietniakOprocha}, \cite{CanovasLopez}, \cite{Banks}, \cite{MaHouLiao}, \cite{ZhangZenLiu}, \cite{CamargoJavierGarcia}, \cite{AcostaGerardoIllanes}, \cite{Fedeli}, \cite{FloresCano}, \cite{RomanFlores}, \cite{BauerSigmund}, and \cite{ACCC}). 

Regarding the shadowing property, it is proved in \cite{FernandezGood} that $f$ has the shadowing property if, and only if, $2^f$ has the shadowing property. This means that shadowing pseudo-orbits of points is an equivalent problem to shadowing pseudo-orbits formed by compact sets. The proof of this result relies on two important facts: the first is that to prove the shadowing property, it is enough to shadow pseudo-orbits with a finite length, and the second is the fact that every compact set can be approximated (in the Hausdorff topology) by compact sets with a finite number of points. Thus, the authors in \cite{FernandezGood} approximate each pseudo-orbit of compact sets with finite length by a finite number of pseudo-orbits of points and use the shadowing property of $f$ to conclude the shadowing property of $2^f$. We note that, in this case, the shadowing compact sets obtained are finite.

In this article, we prove a similar result for the L-shadowing property. The L-shadowing property was introduced in \cite{carvalho2019positively} as a mixture of the well-known shadowing and limit-shadowing properties (see \cites{MR1442020,MR2350246,CK,C,MR3750227} for more information on limit shadowing). While in the shadowing property $\delta$-pseudo-orbits are $\eps$-shadowed, in the limit-shadowing property, pseudo-orbits with errors converging do zero, called limit-pseudo-orbits, are shadowed in the limit, i.e., with the difference between the shadowing orbit and the pseudo-orbit converging to zero. The L-shadowing property both $\eps$-shadows and limit shadows sequences that are simultaneously $\delta$-pseudo-orbits and limit pseudo-orbits (see Definition \ref{Lshad}). This property is stronger than both shadowing and limit-shadowing and implies some of the well-known features of hyperbolic systems, such as a Spectral Decomposition of the non-wandering set (see \cite{artigue2020beyond}*{Theorems B and F}). 

The known examples of homeomorphisms with the L-shadowing property are: the topologically hyperbolic homeomorphisms (expansive homeomorphisms with the shadowing property \cite{AOKIHIRAIDE}), the structurally stable diffeomorphisms (which have shadowing and an expansive non-wandering set \cite{ManeStab}), the continuum-wise hyperbolic homeomorphisms (which are continuum-wise expansive and have shadowing \cite{artigue2024continuum}), the shift map in the Hilbert cube (which is an example in an infinite dimensional metric space \cites{MR4954579,AcostaGerardoIllanes}), and an example without periodic points discussed in \cite{CK} and \cite{artigue2020beyond}. The following is the main result of this article.

\begin{mainthm}\label{hyplshad}
If $f\colon X\rightarrow X$ is a homeomorphism in a compact metric space $X$, then $f$ has the L-shadowing property if, and only if, $2^f$ has the L-shadowing property.
\end{mainthm}

Among the difficulties that appear when proving the L-shadowing property for $2^f$ are the need to deal with pseudo-orbits with infinite length, since we are considering limit pseudo-orbits and shadowing orbits in the limit, and the lack of uniformity in the rate of convergence to zero in these limit-shadowing orbits. We cannot expect to always limit-shadow limit-pseudo-orbits for $2^f$ with finite sets, as it is done in the case of $\delta$-pseudo-orbits and the shadowing property, since this would imply that the elements of the pseudo-orbit converge in the Hausdorff topology to a finite set, which of course does not happen in all cases. Thus, when the shadowing compact set is infinite, the lack of uniformity in contractions on local stable/unstable sets could complicate the calculation of the Hausdorff distance between its iterates and the limit pseudo-orbit.

To deal with these problems, we first use the following characterization of the L-shadowing property.

\begin{theorem}\label{caracterizacaolshadowing}\cite[Theorem E]{artigue2020beyond}
    A homeomorphism, defined on a compact metric space, has the L-shadowing property if, and only if, it has the shadowing property and satisfies the asymptotic local-product-structure: 
    \begin{center} for each $\eps>0$ there is $\delta>0$ such that $d(x,y)\leq\delta$ implies $V^s_\eps(x)\cap V^u_\eps(y)\neq \emptyset$.\end{center}
\end{theorem}

The sets $V^s_\eps(x)$ and $V^u_\eps(y)$ are called the asymptotic $\eps$-stable/$\eps$-unstable sets of $x$, which are defined as the intersection of the $\eps$-stable/$\eps$-unstable with the stable/unstable sets of $x$ (see Defintion \ref{loc}). Asymptotic local-product-structure means, in particular, that we can both $\eps$-shadow and limit shadow pseudo-orbits with only one jump, formed by a past and a future orbit.

The main technical tool we use to deal with these difficulties is the following uniformity in the asymptotic local-product-structure assuming only the L-shadowing property.

\begin{mainthm}\label{uniformity}
    If $f$ has the L-shadowing property, then for each $\eps>0$ there is $\delta>0$ such that for each $r>0$ there exists $k_r\in\mathbb{N}^*$ such that if $d(x,y)<\delta$, then there is $z\in V_\eps^s(x)\cap V^u_\eps(y)$ satisfying 
    $$f^{k_r}(z)\in V^s_r(f^{k_r}(x)) \,\,\,\,\,\, \text{and} \,\,\,\,\,\, f^{-k_r}(z)\in V^u_r(f^{-k_r}(y)).$$
\end{mainthm}

Roughly speaking, this means that there exist points in $V_\eps^s(x)\cap V^u_\eps(y)$ with the same rate of convergence to zero over all pairs $(x,y)\in X\times X$ with $d(x,y)<\delta$. The existence of this uniformity is enough to prove the L-shadowing property for $2^f$ as explained in the proof of Theorem \ref{hyplshad} in Section 2, which also contains the proof of Theorem \ref{uniformity}. In Section 3, we show the lack of uniformity of contraction on local stable/unstable sets on many known examples of homeomorphisms with the L-shadowing property, contrasting whenever possible with the uniformity obtained in Theorem \ref{uniformity}.

\section{L-shadowing for the induced homeomorphism}

In this section, after a few basic definitions, we prove Theorems \ref{hyplshad} and \ref{uniformity}. In what follows, $f\colon X\rightarrow X$ will be a homeomorphism of a compact metric space $(X,d)$.

\begin{definition}[Shadowing]\label{shadowing}
We say that $f$ has the \emph{shadowing property} if given $\varepsilon>0$ there is $\delta>0$ such that for each sequence $(x_k)_{k\in\mathbb{Z}}\subset X$ satisfying
$$d(f(x_k),x_{k+1})<\delta \,\,\,\,\,\, \text{for every} \,\,\,\,\,\, k\in\mathbb{Z}$$ there is $y\in X$ such that
$$d(f^k(y),x_k)<\varepsilon \,\,\,\,\,\, \text{for every} \,\,\,\,\,\, k\in\mathbb{Z}.$$
In this case, we say that $(x_k)_{k\in\mathbb{Z}}$ is a $\delta-$pseudo orbit of $f$ and that $(x_k)_{k\in\mathbb{Z}}$ is
$\varepsilon-$shadowed by $y$.
\end{definition}

\begin{definition}[L-shadowing]\label{Lshad}
We say that $f$ has the \textit{L-shadowing property}, if for every $\eps>0$, there exists $\delta>0$ such that for every sequence $(x_k)_{k\in\mathbb{Z}}\subset X$ satisfying
    \[d(f(x_k),x_{k+1})\leq \delta\;\;\;\mbox{for every}\;\;\;k\in\mathbb{Z}\;\;\;\mbox{and}\]
    \[\lim_{|k|\rightarrow\infty}d(f(x_k),x_{k+1})=0,\] there is $z\in X$ satisfying
    \[d(f^k(z),x_{k})\leq \eps\;\;\;\mbox{for every}\;\;\;k\in\mathbb{Z}\;\;\;\mbox{and}\]
    \[\lim_{|k|\rightarrow\infty}d(f^k(z),x_{k})=0.\] In this case, we say that $(x_k)_{k\in\mathbb{Z}}$ is a \textit{$\delta$-limit-pseudo-orbit} of $f$ and $(x_k)_{k\in\mathbb{Z}}$ is \textit{$\eps$-limit-shadowed} by $z$.
    
\end{definition}

\begin{definition}[Local stable/unstable sets]\label{loc}
For each $x\in X$ and $c>0$, let 
$$W^s_{c}(x):=\{y\in X; \,\, d(f^k(y),f^k(x))\leq c \,\,\,\, \textrm{for every} \,\,\,\, k\geq 0\}$$
be the \emph{c-stable set} of $x$ and
$$W^u_{c}(x):=\{y\in X; \,\, d(f^k(y),f^k(x))\leq c \,\,\,\, \textrm{for every} \,\,\,\, k\leq 0\}$$
be the \emph{c-unstable set} of $x$. We consider the \emph{stable set} of $x\in X$ as the set 
$$W^s(x):=\{y\in X; \,\, d(f^k(y),f^k(x))\to0 \,\,\,\, \textrm{when} \,\,\,\, k\to\infty\}$$
and the \emph{unstable set} of $x$ as the set 
$$W^u(x):=\{y\in X; \,\, d(f^k(y),f^k(x))\to0 \,\,\,\, \textrm{when} \,\,\,\, k\to-\infty\}.$$ Denote by $V^s_{\eps}(x)$ the intersection $W^s(x)\cap W^s_{\eps}(x)$ of the stable and the local stable sets of $x$ and by $V^u_{\eps}(x)$ the intersection $W^u(x)\cap W^u_{\eps}(x)$ of the unstable and the local unstable sets of $x$.
\end{definition}

\begin{proof}[Proof of Theorem \ref{uniformity}]
Let $f$ be a homeomorphism satisfying the L-shadowing property and suppose that the thesis is not true. Then there is $\eps>0$ such that for each $\delta>0$ there exists $r>0$ such that for each $n\in\mathbb{N}$ there is $m_n\in\mathbb{N}$ with $m_n>n$ and there are $x_n,y_n\in X$ with $d(x_n,y_n)<\delta$ satisfying: if $z\in V_\eps^s(x_n)\cap V^u_\eps(y_n)$, then either
\begin{equation}\label{eq1}
    d(f^{m_n}(z),f^{m_n}(x_n))>r \;\;\mbox{or}\;\;d(f^{-m_n}(z),f^{-m_n}(y_n))>r.
\end{equation}
Notice that this implies, in particular, that 
    $$V_r^s(x_n)\cap V_r^u(y_n)=\emptyset \,\,\,\,\,\, \text{for every} \,\,\,\,\,\, n\in\N.$$ For the previous $\eps>0$, consider $\delta\in(0,\eps/3)$ given by the L-shadowing property for $\eps/3$ and let $r\in(0,\frac{\eps}{3})$, $(m_n)_{n\in\mathbb{N}}$ and $(x_n,y_n)_{n\in\mathbb{N}}\subset X\times X$ be as above. Theorem \ref{caracterizacaolshadowing} ensures the existence of $\delta'\in(0,r)$ such that 
    \[d(a,b)\leq \delta' \;\;\Rightarrow\;\; V_r^s(a)\cap V_r^u(b)\neq\emptyset.\] Since $V_r^s(x_n)\cap V_r^u(y_n)=\emptyset$ for every $n\in\N$, it follows that $$d(x_n,y_n)>\delta' \,\,\,\,\,\, \text{for every} \,\,\,\,\,\, n\in\N.$$
    Let $(x,y)\in X\times X$ be an accumulation point of the sequence $(x_n,y_n)_{n\in\mathbb{N}}$. Note that $\delta'\leq d(x,y)\leq\delta$ since $\delta'<d(x_n,y_n)<\delta$ for every $n\in\N$. We will prove that 
    $$V^s_{\frac{\eps}{3}}(x)\cap V^u_{\frac{\eps}{3}}(y)=\emptyset,$$ contradicting the choice of $\delta$ given by the L-shadowing property. If there exists $z\in V^s_{\frac{\eps}{3}}(x)\cap V^u_{\frac{\eps}{3}}(y)$, then there exists $k\in\N$ such that
    $$d(f^k(z),f^k(x))<\frac{\delta'}{2}\,\,\,\,\,\, \text{and} \,\,\,\,\,\, d(f^{-k}(z),f^{-k}(y))<\frac{\delta'}{2}.$$ Choose $n\in\N$ sufficiently large such that $m_n>k$ and that
    $$d(f^i(x_n),f^i(x))<\frac{\delta'}{2} \,\,\,\,\,\, \text{and} \,\,\,\,\,\, d(f^{-i}(y_n),f^{-i}(y))<\frac{\delta'}{2}$$ for every $i\in\{0,\dots,k\}$. Thus, $$d(f^k(z),f^k(x_n))\leq d(f^k(z),f^k(x))+d(f^k(x),f^k(x_n))<\delta'$$ and $$d(f^{-k}(z),f^{-k}(y_n))\leq d(f^{-k}(z),f^{-k}(y))+d(f^{-k}(y),f^{-k}(y_n))<\delta'.$$
    This ensures that the sequence $(w_i)_{i\in\Z}\subset X$ defined by 
    $$w_i=\begin{cases}
    f^{i}(x_n) & \textit { if } i\geq k\\
    f^i(z) & \textit { if } i\in\{-k,\dots,k-1\}\\
    f^i(y_n) & \textit { if } i<-k
\end{cases}$$
is a $\delta'$-limit-pseudo-orbit. The L-shadowing property ensures the existence of $z'\in X$ that $r$-shadows and limit-shadows $(w_i)_{i\in\Z}$. In particular, $z'\in V_{\eps}^s(x_n)\cap V_{\eps}^u(y_n)$ since for $i\in\{0,\dots,k-1\}$ we have
\begin{eqnarray*}
d(f^i(z'),f^i(x_n))&\leq& d(f^i(z'),f^i(z))+d(f^i(z),f^i(x))+d(f^i(x),f^i(x_n))\\
&<&r+\frac{\eps}{3}+\frac{\delta'}{2}<\eps,
\end{eqnarray*}
for $i\in\{-k,\dots,-1\}$ we have
\begin{eqnarray*}
d(f^i(z'),f^i(y_n))&\leq& d(f^i(z'),f^i(z))+d(f^i(z),f^i(y))+d(f^i(y),f^i(y_n))\\
&<&r+\frac{\eps}{3}+\frac{\delta'}{2}<\eps,
\end{eqnarray*}
$$f^k(z')\in V^s_{r}(f^k(x_n)) \,\,\,\,\,\, \text{and} \,\,\,\,\,\, f^{-k}(z')\in V^u_{r}(f^{-k}(y_n)).$$
But since $m_n>k$, this ensures that 
\begin{equation}\label{eq2}
d(f^{m_n}(z'),f^{m_n}(x_n))\leq r \;\;\mbox{and}\;\;d(f^{-m_n}(z'),f^{-m_n}(y_n))\leq r
\end{equation}
contradicting inequalities (\ref{eq1}).
\end{proof}

Using this uniformity result, we can prove Theorem \ref{hyplshad}, that is, L-shadowing for $f$ is necessary and sufficient for $2^f$ to have L-shadowing. First, we stablish the following notation. 

\begin{notation}
For each $\eps>0$, $\underline{k}=(k_r)_{r>0}\subset\N$ with $k_r\to+\infty$ when $r\to0^{+}$, and $x,y\in X$, we define the set $\Gamma_{\eps,{\underline{k}}}(x,y)$ as the following set:
$$\left\{z\in V_\eps^s(x)\cap V^u_\eps(y); \; f^{k_r}(z)\in V^s_r(f^{k_r}(x)) \,\,\,\text{and} \,\,\,
f^{-k_r}(z)\in V^u_r(f^{-k_r}(y)) \,\,\,\forall \,r>0\right\}.$$
\end{notation}

\begin{remark}\label{gammaseq}
Using this notation, we note that Theorem \ref{uniformity} proves that for each $\eps>0$ there is $\delta>0$ and $\underline{k}=(k_r)_{r>0}\subset\N$ with $k_r\to+\infty$ when $r\to0^{+}$ such that if $d(x,y)<\delta$, then there is $z\in \Gamma_{\eps,{\underline{k}}}(x,y)$.
\end{remark}

\begin{proof}[Proof of Theorem \ref{hyplshad}]
First, suppose that $2^f$ has the L-shadowing property. So, given $\eps>0$ there is $\delta>0$ such that every $\delta$-limit-pseudo-orbit of $2^f$ is ${\eps}$-limit-shadowed. If $(x_i)_{i\in\mathbb{Z}}$ is a $\delta$-limit-pseudo-orbit of $f$, then the sequence $(\{x_i\})_{i\in\mathbb{Z}}\subset 2^X$ is a $\delta$-limit-pseudo-orbit of $2^f$. By the L-shadowing property of $2^f$, there is $A\in2^X$ such that 
    \[d_H(f^i(A),\{x_i\})\leq \eps\;\mbox{\,\,\, for every \,\,\,}i\in\mathbb{Z}\;\;\;\mbox{\,\,\,and}\]
    \[d_H(f^i(A),\{x_i\})\rightarrow 0\;\;\;\mbox{\,\,\,when\,\,\,}\;\;\;|i|\rightarrow \infty.\]
    In particular, for each $y\in A$, we have
    \[d(f^i(y),x_i)\leq \eps\;\mbox{ \,\,\,for every \,\,\,}i\in\mathbb{Z}\;\;\;\mbox{\,\,\,and}\]
     \[d(f^i(y),x_i)\rightarrow 0\;\;\;\mbox{\,\,when\,\,}\;\;\;|i|\rightarrow \infty.\] These imply that any point of $A$ $\eps
     $-limit-shadows the $\delta$-limit-pseudo-orbit $(x_i)_{i\in\mathbb{Z}}$, which proves that $f$ has the L-shadowing property.
     
Now, suppose that $f$ has the L-shadowing property. We will prove that $2^f$ has the L-shadowing property. Since $f$ has the shadowing property, $2^f$ also has it (see \cite{good2016shadowing}*{Theorem 3.2}). Thus, by Theorem \ref{caracterizacaolshadowing} it is enough to prove that given $\varepsilon>0$ there is $\delta>0$ such that if $d_H(A,B)<\delta$, then $V^s_{\eps}(A)\cap V^u_{\eps}(B)\neq\emptyset$. For each $\eps>0$ consider $\delta\in(0,\eps)$, given by Theorem \ref{uniformity}, such that for each $r>0$ there exists $k_r\in\mathbb{N}^*$ such that if $d(x,y)<\delta$, then there is $z\in V_\eps^s(x)\cap V^u_\eps(y)$ satisfying 
    $$f^{k_r}(z)\in V^s_r(f^{k_r}(x)) \,\,\,\,\,\, \text{and} \,\,\,\,\,\, f^{-k_r}(z)\in V^u_r(f^{-k_r}(y)).$$ 
For each $A,B\in 2^X$ such that $d_H(A,B)<\delta$, define $$\Psi=\{(x,y)\in A\times B\;|\;d(x,y)\leq\delta\}.$$ 
For $\underline{k}=(k_r)_{r>0}\subset\N$, let \[C=\{z\in\Gamma_{\eps,{\underline{k}}}(x,y)\;|\; (x,y)\in \Psi\}\] and note that $C$ is non-empty. We shall prove that $$C\in 2^X \,\,\,\,\,\, \text{and} \,\,\,\,\,\, C\in {V}^s_\eps(A)\cap {V}^u_\eps(B).$$ 
First, we prove that $C$ is a compact subset of $X$. Consider $(z_i)_{i\in\mathbb{N}}$ a sequence in $C$ with $z_i\rightarrow z$ when $i\to\infty$. We want to prove that $z$ is in $C$, that is, 
$$z\in\Gamma_{\eps,{\underline{k}}}(x,y) \,\,\,\,\,\, \text{for some} \,\,\,\,\,\, (x,y)\in \Psi.$$ 
For each $i\in\mathbb{N}$, $z_i\in\Gamma_{\eps,{\underline{k}}}(x_i,y_i)$ for some $x_i\in A$ and $y_i\in B$ with $d(x_i,y_i)\leq \delta$, since $z_i\in C$. By compactness of $A$ and $B$, we can suppose that $x_i\rightarrow x\in A$ and $y_i\rightarrow y\in B$ when $i\to\infty$, taking a subsequence if necessary. For each $i,n\in\mathbb{N}$, we have
\[d(f^n(z),f^n(x))\leq d(f^n(z),f^n(z_i))+d(f^n(z_i),f^n(x_i))+d(f^n(x_i),f^n(x)).\] Continuity of $f$ ensures that
\[\lim_{i\rightarrow \infty}d(f^n(z),f^n(z_i))=0,\;\;\;\mbox{since }z_i\rightarrow z, \,\,\, \text{and}\]
\[\lim_{i\rightarrow \infty}d(f^n(x_i),f^n(x))=0,\;\;\;\mbox{since }x_i\rightarrow x,\] while
\[d(f^n(z_i),f^n(x_i))\leq \eps,\;\;\;\mbox{ since }z_i\in W^s_\eps(x_i).\] Then, 
\[d(f^n(z),f^n(x))\leq\eps \,\,\,\,\,\, \text{for every} \,\,\,\,\,\, n\in\N,\] which implies that $z\in W^s_\eps(x)$. The proof of $z\in W^u_\eps(y)$ is analogous. To conclude that $C$ is compact, it remains to prove that $z\in\Gamma_{\eps,{\underline{k}}}(x,y)$. For each $r>0$ and $n\geq k_r$, we have
\begin{eqnarray*}
d(f^n(z),f^n(x))&=& d(f^n(\lim_{i\rightarrow\infty}z_i),f^n(\lim_{i\rightarrow\infty} x_i))\\
&=&\lim_{i\rightarrow\infty}d(f^n(z_i),f^n(x_i))<r
\end{eqnarray*}
and
\begin{eqnarray*}
d(f^{-n}(z),f^{-n}(y))&=& d(f^{-n}(\lim_{i\rightarrow\infty} z_i),f^{-n}(\lim_{i\rightarrow\infty} y_i))\\
&=&\lim_{i\rightarrow\infty}d(f^{-n}(z_i),f^{-n}(y_i))<r.
\end{eqnarray*}
This ensures that $z\in\Gamma_{\eps,{\underline{k}}}(x,y)$ and we conclude that $z\in C$ and $C\in 2^X$.

To prove that $C\in V_{\eps}^s(A)\cap V_{\eps}^u(B)$, we first prove that $C\in W_{\eps}^s(A)\cap W_{\eps}^u(B)$. Indeed, for each $z\in C$ there is $(x,y)\in \Psi$ such that $$z\in V_\eps^s(x)\cap V_\eps^u(y),$$ which ensures in particular that 
$$d(f^n(z),f^n(x))\leq\eps \,\,\,\,\,\, \text{and} \,\,\,\,\,\, d(f^{-n}(z),f^{-n}(y))\leq\eps$$
for every $n\in\N$, and consequently
$$f^n(C)\subset B(f^n(A),\eps) \,\,\,\,\,\, \text{and} \,\,\,\,\,\, f^{-n}(C)\subset B(f^{-n}(B),\eps)$$
for every $n\in\N$. Also, for each $a\in A$, there is $b\in B$ such that $d(a,b)\leq \delta$, since $d_H(A,B)<\delta$, and then there is $z\in \Gamma_{\eps,{\underline{k}}}(a,b)$ as noted in Remark \ref{gammaseq}. 
Thus, $z\in C$ and
$$d(f^n(z),f^n(a))\leq\eps \,\,\,\,\,\, \text{for every} \,\,\,\,\,\, n\in\N,$$
which ensures that $f^n(A)\subset B(f^n(C),\eps)$ for every $n\in\N$. A similar argument proves that $f^{-n}(B)\subset B(f^{-n}(C),\eps)$ for every $n\in\N$. These inclusions ensure that $$d_H(f^n(C),f^n(A))\leq\eps \,\,\,\,\,\, \text{and} \,\,\,\,\,\, d_H(f^{-n}(C),f^{-n}(B))\leq\eps$$
for every $n\in\N$ and we conclude that $C\in W_{\eps}^s(A)\cap W_{\eps}^u(B)$. 
Now, we prove that $$C\in W^s(A)\cap W^u(B).$$ Since each $z\in C$ satisfies $z\in \Gamma_{\eps,{\underline{k}}}(x,y)$ for some
$(x,y)\in \Psi$, it follows that for each $r>0$ we have
$$d(f^n(z),f^n(x))\leq r \,\,\,\,\,\, \text{for every} \,\,\,\,\,\, n\geq k_r.$$
This ensures that
$$f^n(C)\subset B(f^n(A),r)\,\,\,\,\,\, \text{for every} \,\,\,\,\,\, n\geq k_r.$$ Also, $f^n(A)\subset B(f^n(C),r)$ for every $n\geq k_r$ since for each $a\in A$ there is $b\in B$ such that $d(a,b)<\delta$ and $z\in \Gamma_{\eps,{\underline{k}}}(a,b)$, which ensures in particular that $z\in C$ and
$$d(f^n(a),f^n(z))<r \,\,\,\,\,\, \text{for every} \,\,\,\,\,\, n\geq k_r.$$
Thus, we conclude that
$$d_H(f^n(C),f^n(A))\leq r\,\,\,\,\,\, \text{for every} \,\,\,\,\,\, n\geq k_r.$$
Similarly, we obtain
$$d_H(f^{-n}(C),f^{-n}(B))\leq r\,\,\,\,\,\, \text{for every} \,\,\,\,\,\, n\geq k_r.$$
This ensures that $C\in W^s(A)\cap W^u(B)$ and finishes the proof.
\end{proof}

\section{Examples}

In this section, we show the lack of uniformity of contraction on local stable/unstable sets and contrast it with the uniformity obtained in Theorem \ref{uniformity} on many known examples of homeomorphisms with the L-shadowing property, which are the topologically hyperbolic homeomorphisms, the structurally stable diffeomorphisms, the continuum-wise hyperbolic homeomorphisms, the shift map in the Hilbert cube, and an example without periodic points discussed in \cite{CK} and \cite{artigue2020beyond}, explaining the differences that appear in each case.


\vspace{+0.4cm}

\hspace{-0.45cm}\textbf{Topologically hyperbolic homeomorphisms}. Topological hyperbolicity is defined as the presence of the shadowing property together with expansiveness, defined as follows.

\begin{definition}[Expansiveness]
For each $x\in X$ and $c>0$ let $$\Gamma_{c}(x)=W^u_{c}(x)\cap W^s_{c}(x)$$ be the dynamical ball of $x$ with radius $c$. We say that $f$ is \emph{expansive} if there exists $c>0$ such that $$\Gamma_c(x)=\{x\} \,\,\,\,\,\, \text{for every} \,\,\,\,\,\, x\in X.$$ A constant $c$ satisfying the above definition is called an expansivity constant.
\end{definition}

Expansiveness was introduced by Utz in \cite{Utz}, explored by many authors since then, and gives information on topological and statistical aspects of chaotic systems. Many features of expansive homeomorphisms are known to these days and we invite the authors to read the monograph \cite{AOKIHIRAIDE} on this subject. One of these features is the uniform contraction on local stable/unstable sets defined below.

\begin{definition}[Uniform contractions]
We say that a homeomorphism $f\colon X\to X$ has \emph{uniform contractions in its local stable/unstable sets} if there exists $\eps>0$ such that for each $r>0$ there exists $k_r\in\N$ satisfying
\begin{equation}\label{Mane}f^n(W^s_\eps(x))\subset W^s_r(f^n(x)) \,\,\,\,\,\, \text{and} \,\,\,\,\,\, f^{-n}(W^u_\eps(x))\subset W^u_r(f^{-n}(x))
\end{equation}
for every $x\in X$ and $n\geq k_r$.
\end{definition}

\begin{proposition}
If $f\colon X\to X$ is an expansive homeomorphism, then $f$ has uniform contractions on its local stable/unstable sets.
\end{proposition}

\begin{proof}
If $\eps>0$ is such that $2\eps$ is an expansivity constant of $f$, then given $r>0$, we can choose $k_r=N\in\N$ from Lemma I of \cite{mane1979expansive} satisfying the above definition.
\end{proof}
Note that, in particular, we have the following equalities for every $x\in X$:
$$W^s_\eps(x)=V^s_{\eps}(x) \,\,\,\,\,\, \text{and} \,\,\,\,\,\, W^u_\eps(x)=V^u_{\eps}(x).$$

\begin{definition}[Local-product-structure]
We say that $f$ has the local-product-structure if for each $\eps>0$ there exists $\delta>0$ such that $d(x,y)<\delta$ implies $W^s_{\eps}(x)\cap W^u_{\eps}(y)\neq\emptyset$.
\end{definition}

If $f$ has the shadowing property, then it satisfies, in particular, the local-product-structure since the later means that we can $\eps$-shadow the $\delta$-pseudo-orbits with one jump. In the expansive case, these properties are indeed equivalent. 

\begin{definition}[Uniform local-product-structure]\label{unif}
We say that $f$ has a \emph{uniform local-product-structure} if for each $\eps>0$ there exist $\delta>0$ and $\underline{k}=(k_r)_{r>0}\subset\N$ such that if $d(x,y)<\delta$, then $W^s_{\eps}(x)\cap W^u_{\eps}(y)\neq\emptyset$ and $W^s_{\eps}(x)\cap W^u_{\eps}(y)\subset \Gamma_{\eps,\underline{k}}(x,y)$.
\end{definition}

This means, in particular, that the conclusion of Theorem \ref{uniformity} holds for every point in $W^s_{\eps}(x)\cap W^u_{\eps}(y)$. We conclude proving the uniform local-product-structure for topologically hyperbolic homeomorphisms.

\begin{proposition}
If $f\colon X\to X$ is an expansive homeomorphism with the local-product-structure, then $f$ has a uniform local-product-structure.
\end{proposition}

\begin{proof}
For each $\eps>0$, choose $\delta>0$ given by the local-product-structure and $\underline{k}=(k_r)_{r>0}\subset\N$ as above (Lemma I in \cite{mane1979expansive}). It follows that
$$W^s_{\eps}(x)\cap W^u_{\eps}(y)=V^s_{\eps}(x)\cap V^u_{\eps}(y)\subset \Gamma_{\eps,\underline{k}}(x,y)$$
for every $x,y\in X$ such that $d(x,y)<\delta$.
\end{proof}

\vspace{+0.4cm}

\hspace{-0.4cm}\textbf{Structurally stable diffeomorphisms}. Structural stability is a fundamental property in dynamical systems, which means that the qualitative behavior of orbits is unaffected by small perturbations of the system. Structurally stable diffeomorphisms admit the shadowing property (see \cite{CRob}) and have an expansive (hyperbolic) non-wandering set (see \cite{ManeStab}). The latter is a consequence of the celebrated solution of the Stability Conjecture. This is enough to ensure the L-shadowing property, as proved in \cite{artigue2020beyond}*{Theorem 4.1}. This class includes a wider range of systems that are not considered in the topologically hyperbolic class, as, for example, the Morse-Smale diffeomorphisms.

Let $\Omega(f)$ denote the set of non-wandering points of $f$, which are the points $x\in X$ such that for each $\eta>0$ there exist $y\in X$ with $d(x,y)<\eta$ and $k\in\N$ such that $d(f^k(y),x)<\eta$. It is proved in \cite{artigue2020beyond}*{Proposition 4.2} that if $\eps>0$ is an expansiveness constant of $f$ restricted to its non-wandering set, then $$W^s_{\eps}(x)\subset W^s(x) \,\,\,\,\,\, \text{and} \,\,\,\,\,\, W^u_{\eps}(x)\subset W^u(x) \,\,\,\,\,\, \text{for every} \,\,\,\,\,\, x\in X.$$ This is enough to ensure that
$$W^s_\eps(x)=V^s_{\eps}(x) \,\,\,\,\,\, \text{and} \,\,\,\,\,\, W^u_\eps(x)=V^u_{\eps}(x)$$
for every $x\in X$. Also, expansiveness in the non-wandering set means that there is an uniform contraction on local stable/unstable sets inside the non-wandering set. In the following result, we prove that there is a uniform contraction in the local stable/unstable sets of points in $\Omega(f)$ (the local stable/unstable sets are not necessarily contained in $\Omega(f)$).

\begin{proposition}
If a homeomorphism $f\colon X\to X$ has the shadowing property and an expansive non-wandering set, then there exists $\eps>0$ such that for each $r>0$ there exists $k_r\in\N$ satisfying
$$f^n(W^s_\eps(x))\subset W^s_r(f^n(x)) \,\,\,\,\,\, \text{and} \,\,\,\,\,\, f^{-n}(W^u_\eps(x))\subset W^u_r(f^{-n}(x))
$$
for every $x\in \Omega(f)$ and $n\geq k_r$.
\end{proposition}

\begin{proof}Since $f_{|\Omega(f)}$ is topologically hyperbolic, the Spectral Decomposition Theorem ensures that $\Omega(f)$ splits into a finite union of disjoint basic sets $C_1,\dots,C_k$, which are compact, invariant, and transitive.
Let $\eps>0$ be an expansiveness constant of $f$ restricted to its non-wandering set and assume, in addition, that the $\eps$-neighborhood of all basic sets of $f$ are disjoint. Note that each $C_i$ is the maximal invariant set inside their $2\eps$-neighborhood since $C_i$ is a chain-recurrent class. Now by contradiction, assume that there exists $r>0$ such that for each $n\in\N$ there exists $x_n\in\Omega(f)$, $y_n\in W^s_{\eps}(x_n)$, and $k_n\geq n$ such that 
$$d(f^{k_n}(y_n),f^{k_n}(x_n))>r.$$ 
We can assume, by choosing a subsequence, that $(x_n)_{n\in\N}$ is contained in a basic set $C_i$. Thus, if $x,y\in X$ are accumulation points of the sequences $(f^{k_n}(x_n))_{n\in\N}$ and $(f^{k_n}(y_n))_{n\in\N}$, respectively, then $x,y\in C_i\subset\Omega(f)$, $d(x,y)\geq r$, and $y\in\Gamma_{\eps}(x)$, contradicting the fact that $\Omega(f)$ is expansive.
\end{proof}

However, we prove below that such uniform contraction only holds in the whole space when $f$ is asymptotically expansive.

\begin{definition}[Asymptotic expansiveness]
We say that a homeomorphism $f\colon X\to X$ is \emph{asymptotically expansive} if there exists $\eps>0$ such that $$V^s_{\eps}(x)\cap V^u_{\eps}(x)=\{x\} \,\,\,\,\,\, \text{for every} \,\,\,\,\,\, x\in X.$$ The set $V^s_{\eps}(x)\cap V^u_{\eps}(x)$ is called the asymptotic dynamical ball of $x$ of radius $\eps$.
\end{definition}

This notion was first defined in \cite{MR4439460}*{Definition 1}. For a homeomorphism with the L-shadowing property, expansiveness and asymptotic expansiveness are equivalent, as noted in \cite{MR4439460}*{Remark 2.6}. This is actually a consequence of \cite{artigue2020beyond}*{Lemma 3.3}. 

\begin{proposition}
If a homeomorphism $f\colon X\to X$ has uniform contractions in its local stable/unstable sets, then $f$ is asymptotically expansive.
\end{proposition}

\begin{proof}
Let $\eps>0$ be given by the uniform contraction in the local stable/unstable sets of $f$ and assume by contradiction the existence of $z\in X$ and 
$$y\in V^s_{\eps}(z)\cap V^u_{\eps}(z)\setminus\{z\}.$$ Let $r=d(z,y)>0$ and choose $k_{\frac{r}{2}}\in\N$ such that
$$f^n(W^s_\eps(x))\subset W^s_{\frac{r}{2}}(f^n(x)) \,\,\,\,\,\, \text{and} \,\,\,\,\,\, f^{-n}(W^u_\eps(x))\subset W^u_{\frac{r}{2}}(f^{-n}(x))$$
for every $x\in X$ and $n\geq k_{\frac{r}{2}}$. Thus, if $y'=f^{-k_r}(y)$ and $z'=f^{-k_r}(z)$, then $y'\in V^s_{\eps}(z')\cap V^u_{\eps}(z')\setminus\{z'\}$ and
$$d(f^{k_{\frac{r}{2}}}(y'),f^{k_{\frac{r}{2}}}(z'))=d(z,y)=r>\frac{r}{2},$$
contradicting the choice of $k_{\frac{r}{2}}$.
\end{proof}

Thus, homeomorphisms with shadowing and an expansive non-wandering set that are not expansive (in the whole space) are necessarily not asymptotically expansive and, consequently, do not have uniform contractions in its local stable/unstable sets. This illustrates the importance of Theorem \ref{uniformity} which ensures the existence of points in $\Gamma_{\eps,\underline{k}}(x,y)$
for every $x,y\in X$ with $d(x,y)<\delta$, even when there is no uniform contraction on the whole set $V^s_{\eps}(x)\cap V^u_{\eps}(y)$.

\vspace{+0.4cm}

\hspace{-0.4cm}\textbf{Continuum-wise hyperbolic homeomorphisms}. The continuum-wise hyperbolic homeomorphisms are defined by considering the generalization of expansiveness, called continuum-wise expansiveness, and a form of local-product-structure that is weaker than the one existing in the hyperbolic case, called cw-local-product-structure. 

\begin{definition}[Continuum-wise expansiveness]
We say that a homeomorphism $f\colon X\to X$ of a compact metric space is \emph{continuum-wise expansive}, or simply cw-expansive, if there exists $c>0$ such that $\Gamma_{c}(x)$ is totally disconnected for every $x\in X$. The number $c$ is called a \emph{cw-expansivity constant} of $f$.
\end{definition}

Cw-expansiveness was first considered by Kato in \cite{Kato2} and \cite{Kato1} where some of the chaotic properties, as well as some examples, of cw-expansive homeomorphisms were discussed.

\begin{definition}[Local stable/unstable continua]
We denote by $C^s_c(x)$ the $c$-stable continuum of $x$, that is, the connected component of $x$ in $W^s_{c}(x)$, and by $C^u_c(x)$ the $c$-unstable continuum of $x$, that is, the connected component of $x$ in $W^u_{c}(x)$.
\end{definition}

In the following lemma, which is based on \cite{Kato1}*{Proposition 2.1}, we obtain uniform contractions in the local stable/unstable continua of a cw-expansive homeomorphism.

\begin{lemma}\label{LemaManeversaocw}
If $f\colon X\to X$ is a cw-expansive homeomorphism of a compact metric space and $2\eps$ is a cw-expansive constant of $f$, then for each $r>0$ there is $k_r\in\N$ such that \[\diam(f^n(C_\varepsilon^s(x)))\leq r\,\,\,\,\,\, \text{and} \,\,\,\,\,\, \diam(f^{-n}(C^u_\eps(x)))\leq r\] for every $x\in X$ and $n\geq k_r$.
\end{lemma}
\begin{proof}
We prove the statement for the stable case and note that the unstable case follows in a similar way. Suppose by contradiction that there exists $r>0$ such that for each $n\in\N$ there exist $x_n\in X$ and $k_n\geq n$ such that $\diam(f^{k_n}(C_\varepsilon^s(x_n)))>r$. Let $C$ be a continuum that is the limit of a subsequence $(f^{k_{n_j}}(C_\varepsilon^s(x_{n_j})))_{j\in\N}$. Then $C$ is a continuum with $\diam(C)\geq r>0$ that is both $2\eps$-stable and $2\eps$-unstable, since for each $i\in\Z$ we can choose $j\in\N$ sufficiently large so that $i+k_{n_j}>0$ and conclude that
$$\diam(f^i(C))=\lim_{j\to\infty}\diam(f^i(f^{k_{n_j}}(C_\varepsilon^s(x_{n_j}))))\leq2\eps.$$
This contradicts that $2\eps$ is a cw-expansive constant of $f$ and finishes the proof.
\end{proof}
Note that this proves, in particular, that
$$C^s_\eps(x)\subset V^s_{\eps}(x) \,\,\,\,\,\, \text{and} \,\,\,\,\,\, C^u_\eps(x)\subset V^u_{\eps}(x)$$
for every $x\in X$, which in turn implies
$$C^s_\eps(x)\cap C^u_\eps(y)\subset V^s_{\eps}(x)\cap V^u_{\eps}(y)$$ for every $x,y\in X$.

\begin{definition}[Continuum-wise hyperbolicity]
We say that $f$ satisfies the $cw$-local-product-structure if for each $\eps>0$ there exists $\delta>0$ such that $$C^s_\eps(x)\cap C^u_\eps(y)\neq \emptyset \,\,\,\,\,\, \text{ whenever } \,\,\,\,\,\, d(x,y)< \delta.$$ The $cw$-expansive homeomorphisms satisfying the $cw$-local-product-structure are called $cw$-hyperbolic.
\end{definition}

The class of cw-hyperbolic homeomorphisms contains all topologically hyperbolic homeomorphisms, when defined on a Peano continuum, and also the Walter's pseudo-Anosov diffeomorphism of the Sphere $\mathbb{S}^2$ \cite{MR648108}*{Example 1, p. 140}. For more information on cw-hyperbolicity, see \cite{ArrudaCarvalhoSarmiento2024}, \cite{artigue2024continuum}, \cite{CD}, and \cite{CarvalhoRigo2023}.

Now, let $\eps>0$ be such that $2\eps$ is a cw-expansivity constant of $f$ and $\delta\in(0,\eps)$ be given by the cw-local-product-structure of $f$ for this $\eps$. Given $r>0$, we can choose $k_r=N\in\N$ as in Lemma \ref{LemaManeversaocw} to obtain 
$$C^s_{\eps}(x)\cap C^u_{\eps}(y)\subset \Gamma_{\eps,\underline{k}}(x,y)$$
for every $x,y\in X$ such that $d(x,y)<\delta$, where $\underline{k}=(k_r)_{r>0}$ was chosen above. This means that the conclusion of Theorem \ref{uniformity} holds for every point in $C^s_{\eps}(x)\cap C^u_{\eps}(y)$. 

\begin{remark} We note that the sets $W^s_\eps(x)$ and $C^s_\eps(x)$ (similarly $W^u_\eps(x)$ and $C^u_\eps(x)$) can differ a lot. While, on surfaces, $C^s_\eps(x)$ is an arc (assuming cw$_F$-hyperbolicity \cite{ArrudaCarvalhoSarmiento2024}), $W^s_\eps(x)$ can be non-locally-connected and contain a Cantor set of distinct arcs. At least a countable set of these distinct arcs are in $V^s_\eps(x)$ (this can be seen following the proof of \cite{ArtigueDend}*{Proposition 2.2.2}) which ensures that the set $V^s_{\eps}(x)\cap V^u_{\eps}(y)$ can be much larger than $C^s_{\eps}(x)\cap C^u_{\eps}(y)$. Thus, the uniform contraction explained above could fail on the whole set $V^s_{\eps}(x)\cap V^u_{\eps}(y)$.
\end{remark}

\vspace{+0.4cm}

\hspace{-0.45cm}\textbf{Shift map on the Hilbert cube}. Let $X$ be a compact metric space and consider the backward shift map $\sigma\colon X^{\mathbb{Z}}\rightarrow X^{\mathbb{Z}}$ defined by
\[\sigma((x_i)_{i\in\mathbb{Z}})= (x_{i+1})_{i\in\mathbb{Z}}.\]
If $X$ is a finite set, then the shift map is topologically hyperbolic and is included in the first class of this section. We now discuss the case $X=[0,1]$, that is, $\sigma$ is the shift map in the Hilbert cube. Consider on $[0,1]^\mathbb{Z}$ the following metric: for $x=(x_i)_{i\in\mathbb{Z}},y=(y_i)_{i\in\mathbb{Z}}\in[0,1]^{\mathbb{Z}}$, let \[d(x,y)=\sum_{i\in\mathbb{Z}}\dfrac{|x_i-y_i|}{2^{|i|}}.\] 
We begin proving that the shift map does not have the uniform asymptotic local-product-structure, which is defined similarly to Definition \ref{unif} changing the local stable/unstable sets by the asymptotic local stable/unstable sets.

\begin{proposition}
The shift map $\sigma\colon [0,1]^{\mathbb{Z}}\to [0,1]^{\mathbb{Z}}$ does not have a uniform asymptotic local-product-structure.
\end{proposition}

\begin{proof}
For each $0<\varepsilon<\frac{1}{2}$, $\delta\in(0,\varepsilon/2)$, $\underline{k}=(k_r)_{r>0}\subset\N$ with $k_r\rightarrow\infty$ when $r\to0^+$, $x=(x_i)_{i\in\mathbb{Z}}\in X$, and $y=(y_i)_{i\in\mathbb{Z}}\in X$ with $d(x,y)<\delta$, we exhibit points in $V_\varepsilon^s(x)\cap V_\varepsilon^u(y)$, which do not belong to $\Gamma_{\varepsilon,\underline{k}}(x,y)$. For some fixed $r<\frac{\delta}{2}$ sufficiently small, we can choose $n\in\mathbb{N}$ such that $r<\frac{1}{n}<\frac{\delta}{2}$. For each $m>n$, define the sequence $z^m=(z^m_i)_{i\in\mathbb{Z}}$ by
\[
z^m_i=
\begin{cases}
x_i, & \text{if } i\geq 0 \text{ and } i\neq m,\\[4pt]
x_i +\dfrac{1}{n}, & \text{if } i=m \text{ and } x_i\leq \dfrac{1}{2},\\[6pt]
x_i -\dfrac{1}{n}, & \text{if } i=m \text{ and } x_i>\dfrac{1}{2},\\[6pt]
y_i, & \text{if } i<0 \text{ and } i\neq -m,\\[4pt]
y_i +\dfrac{1}{n}, & \text{if } i=-m \text{ and } y_i\leq \dfrac{1}{2},\\[6pt]
y_i -\dfrac{1}{n}, & \text{if } i=-m \text{ and } y_i>\dfrac{1}{2}.
\end{cases}
\]
Note that $z^m\in V_\eps^s(x)\cap V_\eps^u(y)$ for every $m>n$ since for each $k\in\mathbb{N}$ we have
\[\begin{array}{rcl}
 d(\sigma^k(x),\sigma^k(z^m)) & = & \displaystyle \sum_{i\in\mathbb{Z}}\dfrac{|x_{i}-z_{i}^m|}{2^{|k-i|}}\\
&=& \displaystyle\dfrac{|x_{-m}-y_{-m}\pm\frac{1}{n}|}{2^{k+m}}+\frac{\frac{1}{n}}{2^{|k-m|}}+\sum_{i<0,\;i\neq -m}\dfrac{|x_i-y_i|}{2^{k-i}}\\
&=&\displaystyle\dfrac{|x_{-m}-y_{-m}\pm\frac{1}{n}|}{2^{k+m}}+\dfrac{1}{n2^{|k-m|}}+\dfrac{1}{2^k}\sum_{i<0,\;i\neq -m}\dfrac{|x_i-y_i|}{2^{|i|}}\\
&\leq&\displaystyle\dfrac{\delta}{2^k}+\dfrac{2}{n2^{|k-m|}}+\dfrac{\delta}{2^k}\;\;<\;\;\eps,\\
\end{array}\] 
and, analogously,
\[\begin{array}{rcl}d(\sigma^{-k}(y),\sigma^{-k}(z^m)) &=&\displaystyle\dfrac{|y_{m}-x_{m}\pm\frac{1}{n}|}{2^{|k-m|}}+\dfrac{1}{n2^{|k-m|}}+\dfrac{1}{2^k}\sum_{i\geq0,\;i\neq m}\dfrac{|y_i-x_i|}{2^{|i|}}\\
&\leq&\dfrac{\delta}{2^k}+\dfrac{2}{n2^{|k-m|}}+\displaystyle\dfrac{\delta}{2^k}\;\;<\;\;\eps.
\end{array}\]
These distances converge to zero when $k\to\infty$ since $z^m_i=x_i$ for $i>m$ and $z^m_i=y_i$ for $i<-m$. Also, for each $m>n$ we have
\[d(\sigma^m(x),\sigma^m(z^m))\geq \dfrac{1}{n}>r.\] The last inequality implies that if $k_r>n$, for $m=k_r$, we have
$z^{k_r}\notin \Gamma_{\eps,\underline{k}}(x,y)$.
\end{proof}

Differently from the cw-expansive case, the shift map has local stable/unstable continua that are not stable/unstable (see \cite{MR4954579}*{Remark 3.11}). We prove below the existence of continua inside each local stable/unstable set with uniform diameter and uniform contractions. For each $\eps\in(0,1/4)$ and $x=(x_i)_{i\in\Z}\in [0,1]^{\mathbb{Z}}$, consider \[D^s_{\eps}(x)= \prod_{i<0} \left([x_i-\eps,x_i+\eps]\cap[0,1]\right)\times \prod_{i\geq 0}\{x_i\} \] and  \[D^u_{\eps}(x)= \prod_{i<0} \{x_i\}\times \prod_{i\geq 0}\left([x_i-\eps,x_i+\eps]\cap[0,1]\right). \]
Note that 
    \[\diam(D^s_{\eps}(x)) = \sum_{i<0}\frac{\diam([x_i-\eps,x_i+\eps]\cap[0,1])}{2^{|i|}}\] and \[\diam(D^u_{\eps}(x)) = \sum_{i\geq 0}\frac{\diam([x_i-\eps,x_i+\eps]\cap[0,1])}{2^{|i|}},\]
which ensures that $\eps\leq\diam(D^s_{\eps}(x))\leq4\eps$ and $\eps\leq\diam(D^u_{\eps}(x))\leq4\eps$.

    \begin{proposition}\label{sigmaunif}
        For each $r\in(0,\eps)$, there is $k_r\in\N$ such that \begin{equation}\label{uniformcontraction}
            \diam(\sigma^n(D^s_{\eps}(x)))\leq r\,\,\,\,\,\, \text{and} \,\,\,\,\,\, \diam(\sigma^{-n}(D^u_{\eps}(x)))\leq r
        \end{equation} for every $x\in [0,1]^{\mathbb{Z}}$ and $n\geq k_r$.
    \end{proposition}
    \begin{proof}
    For each $n\in\mathbb{N}$ we have
    \begin{equation}\label{diamDs}
        \diam(\sigma^n(D^s_{\eps}(x))) = \sum_{i<0}\frac{\diam([x_i-\eps,x_i+\eps]\cap[0,1])}{2^{|i|+n}}\leq \dfrac{2\eps}{2^n}=\dfrac{\eps}{2^{n-1}}\end{equation} and
    \begin{equation}\label{diamDu}\diam(\sigma^{-n}(D^u_{\eps}(x))) = \sum_{i\geq 0}\frac{\diam([x_i-\eps,x_i+\eps]\cap[0,1])}{2^{i+n}}\leq \dfrac{4\eps}{2^n}=\dfrac{\eps}{2^{n-2}}.\end{equation} 
For each $r\in(0,\eps)$, consider $k_r\in\N$ such that \begin{equation*}\label{kreq}\dfrac{\eps}{2^{k_r-2}}<r.\end{equation*} Thus, by (\ref{diamDs}) and (\ref{diamDu}), the continua $D^s_{\eps}(x)$ and $D^u_{\eps}(x)$ satisfy (\ref{uniformcontraction}) for every for $x\in [0,1]^{\mathbb{Z}}$ and $n\geq k_r$.
\end{proof}

We now prove that $\sigma$ has a local-product-structure between the local stable/unstable continua $D^s_{\eps}(x)$ and $D^u_{\eps}(y)$. Let $\underline{k}=(k_r)_{r>0}$ be as in Proposition \ref{sigmaunif}.

\begin{proposition}
For each $\eps>0$ there exists $\delta\in(0,\eps)$ such that if $d(x,y)<\delta$, then $D^s_{\eps}(x)\cap D^u_{\eps}(y)\neq\emptyset$. Also, $D^s_{\eps}(x)\cap D^u_{\eps}(y)\subset\Gamma_{\eps,\underline{k}}(x,y)$.
\end{proposition}

\begin{proof} Let $\delta = \varepsilon/4$ and observe that whenever $x,y \in [0,1]^{\mathbb{Z}}$ satisfy $d(x,y) < \delta$, we have 
\[
D^s_{\eps}(x) \cap D^u_{\eps}(y) = \{ z_{(x,y)} \},
\]
where $z_{(x,y)} = (z_i)_{i \in \mathbb{Z}} \in [0,1]^{\mathbb{Z}}$ is defined by 
\[
z_i =
\begin{cases}
x_i, & \text{if } i \ge 0,\\[4pt]
y_i, & \text{if } i < 0.
\end{cases}
\]
Proposition \ref{sigmaunif} ensures that $z_{(x,y)} \in \Gamma_{\varepsilon,\underline{k}}(x,y)$, that is,
\[
D^s_{\eps}(x) \cap D^u_{\eps}(y) \subset \Gamma_{\varepsilon,\underline{k}}(x,y).
\] This finishes the proof.\end{proof}

\vspace{+0.4cm}

\hspace{-0.4cm}\textbf{L-shadowing without periodic points}. The following example was introduced in \cite{CK} and further explored in \cite{artigue2020beyond} where it was proved to have the L-shadowing property. For each $r>1$ consider the set $\{0,\dots,r-1\}$ endowed with the discrete metric $\rho$, let $\Omega_r=\{0,\dots,r-1\}^{\Z}$ and consider in $\Omega_r$ the Tychonoff product topology. Let $p$ and $q$ be relatively prime integers and $X_{(p,q)}$ be the set of all sequences in $\Omega_{p+q-1}$ whose entries are vertices visited during a bi-infinite walk on the directed graph with two loops, one of length $p$ and one of length $q$ as shown in Figure \ref{figGraph-pq}. 
\begin{center}
\begin{figure}[ht]\label{fig}
\includegraphics{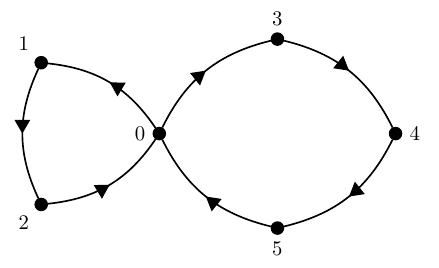}
\caption{A graph presenting the shift space $X_{(3,4)}$.}
\label{figGraph-pq}
\end{figure}
\end{center}
It is clear that $X_{(p,q)}$ is invariant by the shift map on $\Omega_{p+q-1}$ and is a subshift of finite type. 
Let $(p_n)_{n=1}^\infty$ be a strictly increasing sequence of prime numbers. For each $n\in\N$ let $X_n=X_{(p_n,p_{n+1})}$, $\sigma_n$ be the shift transformation on $\Omega_{p_{n}+p_{n+1}-1}$ restricted to $X_n$, and consider the product system $F=\sigma_1\times\sigma_2\times\ldots$ on $X=\prod_{n=1}^\infty X_n$. 
The distance we consider in the space $X$ is defined as follows. First, we clarify some notation: if $a\in X$, then $a=(a_n)_{n\in\N}$ where $a_n\in X_n$ 
for every $n\in\N$, and each $a_n$ is equal to a sequence denoted by $(a_{n,k})_{k\in\Z}$. 
In this way, $a_{n,k}$ denotes the element of position $k$ in the sequence $a_n$, that, in turn, 
is the element of position $n$ of $a$. For each $x,y\in X$, we define
$$d(x,y)=\sum_{n=1}^{+\infty}\frac{d_n(x_n,y_n)}{2^n}$$
where $d_n$ is the metric in $X_n$ defined by $$d_n(x_n,y_n)=\sum_{k\in\Z}\frac{\rho(x_{n,k},y_{n,k})}{2^{|k|}}.$$

\begin{proposition}
The map $F\colon X\to X$ is not asymptotically expansive and does not have uniform contractions on its asymptotic local stable/unstable sets.
\end{proposition}

\begin{proof}
Given $\varepsilon > 0$, let $n_{\varepsilon} \in \mathbb{N}$ be such that $\frac{1}{2^{n_{\varepsilon}}} < \varepsilon$ and choose $r > 0$ sufficiently small so that
\[
r < \frac{1}{2^{\, n_\varepsilon}}\sum_{n=1}^{p_{n_{\varepsilon}}\cdot p_{n_{\varepsilon}-1}}\displaystyle \frac{ 1}{2^n}.
\]
Consider $x = (x_n)_{n \in \mathbb{N}}$ defined by 
\[
x_n = (\ldots, 0,1,\ldots, p_n - 1, 0,1,\ldots, p_n - 1, \ldots),
\qquad x_{n,0} = 0,
\]
that is, the $n$-th coordinate of $x$ is the bi-infinite sequence consisting of repeated loops of period $p_n$ starting and ending at $0$.
For each $i \in \mathbb{N}$, define $z(i) = (z(i)_n)_{n \in \mathbb{N}}$ as follows:
\[
z(i)_n = 
\begin{cases}
x_n, & \text{if } n \neq n_{\varepsilon},\\[4pt]
(z(i)_{n_{\varepsilon},j})_{j \in \mathbb{Z}}, & \text{if } n = n_{\varepsilon},
\end{cases}
\]
where the coordinate $z(i)_{\, n_{\varepsilon}}$ is given by
\[
z(i)_{n_{\varepsilon},j}=
\begin{cases}
x_{n_{\varepsilon},j}, & \text{if } j \le i p_{n_{\eps}},\\[4pt]
p_{n_{\eps}} + 1, & j = i p_{n_{\eps}} + 1,\\
p_{n_{\eps}} + 2, & j = i p_{n_{\eps}} + 2,\\
\vdots & \\
p_{n_{\eps}+1} - 1, & j = i p_{n_{\eps}} + p_{n_{\eps}+1} - 1,\\[4pt]
0, & j = i p_{n_{\eps}} + p_{n_{\eps}+1},\\
p_{n_{\eps}} + 1, & j = i p_{n_{\eps}} + p_{n_{\eps}+1} + 1,\\
p_{n_{\eps}} + 2, & j = i p_{n_{\eps}} + p_{n_{\eps}+1} + 2,\\
\vdots & \\
p_{n_{\eps}+1} - 1, & j = i p_{n_{\eps}} + p_{n_{\eps}+1} + p_{n_{\eps}+1} - 1,\\
\vdots & \\
0, & j = i p_{n_{\eps}} + p_{n_{\eps}}\cdot p_{n_{\eps}+1} - p_{n_{\eps}+1},\\
p_{n_{\eps}} + 1, & j = i p_{n_{\eps}} + p_{n_{\eps}}\cdot p_{n_{\eps}+1} - p_{n_{\eps}+1} + 1,\\
p_{n_{\eps}} + 2, & j = i p_{n_{\eps}} + p_{n_{\eps}}\cdot p_{n_{\eps}+1} - p_{n_{\eps}+1} + 2,\\
\vdots & \\
p_{n_{\eps}+1} - 1, & j = i p_{n_{\eps}} + p_{n_{\eps}}\cdot p_{n_{\eps}+1} - 1,\\[4pt]
x_{n_{\varepsilon},j}, & j \ge i p_{n_{\eps}} + p_{n_{\eps}}\cdot p_{n_{\eps}+1}.
\end{cases}
\]

Note that all coordinates of $z(i)$ and $x$ coincide, except for the $n_{\varepsilon}$-th coordinate. Indeed, while the $n_{\varepsilon}$-th coordinate of $x$ consists solely of loops of period $p_{n_\eps}$ based at $0$, the coordinate $z(i)_{\, n_{\varepsilon}}$ begins by following 
exactly $i$ loops of length $p_{n_\eps}$, then switches to $p_{n_{\eps}}$ loops of length $p_{n_\eps+1}$ and after that
resumes the loops of length $p_{n_\eps}$, matching the subsequent coordinates of $x_{\, n_{\varepsilon}}$.
We will prove that 
$$z(i) \in V_\varepsilon^s(x) \cap V_\varepsilon^u(x) \,\,\,\,\,\, \text{and} \,\,\,\,\,\, d\bigl(F^{i p_{n_\varepsilon}}(x),\, F^{i p_{n_\varepsilon}}(z(i))\bigr) \ge r $$ 
for every $i \in \mathbb{N}$. First, note that for each $i\in\N$ we have
$$d\bigl(x,z(i)\bigr)=\frac{1}{2^{\, n_\varepsilon}}\sum_{j=1}^{p_{n_\eps}\cdot p_{n_\eps+1} - 1}
        \frac{1}{2^{\, |i p_{n_\eps} + j|}}.$$
Also, for each $n \in \mathbb{N}$, we have
\[
d\bigl(F^{n}(x), F^{n}(z(i))\bigr)
=
\frac{1}{2^{\, n_\varepsilon}}
    \displaystyle 
    \sum_{j=1}^{p_{n_\eps}\cdot p_{n_\eps+1} - 1}
        \frac{1}{2^{\, |i p_{n_\eps} + j - n|}},
\]
which is smaller than $\varepsilon$ and converges to $0$ as $n \to \infty$. Similarly, for each $n\in\mathbb{N}$, we have
\[
d\bigl(F^{-n}(x), F^{-n}(z(i))\bigr)
=
\frac{1}{2^{\, n_\varepsilon}}
    \displaystyle 
    \sum_{j=1}^{p_{n_\eps}\cdot p_{n_\eps+1} - 1}
        \frac{1}{2^{\, |i p_{n_\eps} + j + n|}}
< \frac{\varepsilon}{2^{n}}.
\]
This proves that $z(i) \in V_\varepsilon^s(x) \cap V_\varepsilon^u(x)$ for every $i\in\N$. Also, when $n = i p_{n_\eps}$ we obtain
\[
d\bigl(F^{i p_{n_\eps}}(x), F^{i p_{n_\eps}}(z(i))\bigr)
=
\frac{1}{2^{\, n_\varepsilon}}
    \displaystyle 
    \sum_{j=1}^{p_{n_\eps}\cdot p_{n_\eps+1} - 1}
        \frac{1}{2^{j}}
\ge r.
\]
Thus, for any given sequence $\underline{k} = (k_r)_{r>0}$ with $k_r \to \infty$ as $r \to 0^{+}$, we can choose for each $r>0$ an $i_0\in\N$ so that $i_0p_{n_\eps}>k_r$ and
\[
d\bigl(F^{i_0 p_{n_\varepsilon}}(x),\, F^{i_0 p_{n_\varepsilon}}(z(i_0))\bigr) \ge r,
\]
while still having 
\[
z(i_0) \in V_\varepsilon^s(x) \cap V_\varepsilon^u(x).
\]
This completes the proof.     
\end{proof}

Once more, this illustrates the importance of Theorem \ref{uniformity} and concludes the discussion for all listed examples of homeomorphisms with the L-shadowing property.

\vspace{+0.4cm}

\hspace{-0.45cm}\textbf{Acknowledgments:} Bernardo Carvalho was supported by the CNPq project number 446192/2024. Welington Cordeiro was supported by CNPq grant number 173057/2023-3. Mayara Antunes was supported by CAPES project number 88887.015134/2024-00. The authors thank LNCC for hosting a research visit of Mayara Antunes and Welington Cordeiro where part of this article was developed.

\bibliographystyle{plain}

\bibliography{bibliographyACCC}

\vspace{1.5cm}
\noindent

{\em M. Antunes}
\vspace{0.2cm}

Departamento de Ciências Exatas,

Universidade Federal Fluminense - UFF

Avenida dos trabalhadores, 420, Vila Santa Cecília

Volta Redonda - RJ, Brasil.

\vspace{0.2cm}

\email{mayaraantunes@id.uff.br}
\vspace{1.0cm}

{\em B. Carvalho}
\vspace{0.2cm}

\noindent

National Laboratory for Scientific Computing – LNCC/MCTI 

Av. Getúlio Vargas 333, CEP 25651-070, 

Petrópolis – RJ, Brazil

\vspace{0.2cm}

\email{bmcarvalho@lncc.br}

\vspace{1.0cm}
{\em W. Cordeiro}
\vspace{0.2cm}

\noindent

Instituto de Matem\'atica e Computação, IMC - UNIFEI

Av. B P S, 1303 - Centro, 37500-185 

Itajubá - MG, Brazil




\vspace{0.2cm}

\email{welingtonscordeiro@gmail.com }

\end{document}